\documentclass[11pt,reqno,tbtags]{amsart}
\usepackage{amssymb}
\usepackage{natbib}
\bibpunct[, ]{[}{]}{;}{n}{,}{,}

\numberwithin{equation}{section}
\allowdisplaybreaks

\newtheorem*{property*}{Property \csname @currentlabel\endcsname}

\newtheorem{theorem}{Theorem}[section]
\newtheorem{lemma}[theorem]{Lemma}
\newtheorem{proposition}[theorem]{Proposition}
\newtheorem{corollary}[theorem]{Corollary}
\newtheorem{conjecture}[theorem]{Conjecture}

\theoremstyle{definition}

\newtheorem{remark}[theorem]{Remark}

\newtheorem*{ack}{Acknowledgement}

\theoremstyle{remark}

\newenvironment{romenumerate}{\begin{enumerate}
 }{\end{enumerate}}

\newcounter{oldenumi}
{\setcounter{oldenumi}{\value{enumi}}
\begin{romenumerate} \setcounter{enumi}{\value{oldenumi}}}
{\end{romenumerate}}

\newcounter{thmenumerate}
\newenvironment{thmenumerate}
{\setcounter{thmenumerate}{0}%
 \def\item{\par
 \refstepcounter{thmenumerate}\textup{(\roman{thmenumerate})\enspace}}
}
{}

\newcounter{xenumerate}   

\newcommand\pfitem[1]{\par(#1):}

\newcommand{\refT}[1]{Theorem~\ref{#1}}
\newcommand{\refC}[1]{Corollary~\ref{#1}}
\newcommand{\refL}[1]{Lemma~\ref{#1}}
\newcommand{\refR}[1]{Remark~\ref{#1}}
\newcommand{\refS}[1]{Section~\ref{#1}}
\newcommand{\refP}[1]{Proposition~\ref{#1}}

\newcommand{\refApp}[1]{Appendix~\ref{#1}}
\newcommand{\refTab}[1]{Table~\ref{#1}}
\newcommand{\refand}[2]{\ref{#1} and~\ref{#2}}

\begingroup
  \divide\count255 by 60
  \multiply\count255 by -60
  \advance\count255 by \time
  \ifnum \count255 < 10 \xdef\klockan{\the\count1.0\the\count255}
  \else\xdef\klockan{\the\count1.\the\count255}\fi
\endgroup

\DeclareMathOperator*{\sumx}{\sum\nolimits^{*}}

\newcommand\set[1]{\ensuremath{\{#1\}}}

\newcommand\bigpar[1]{\bigl(#1\bigr)}
\newcommand\Bigpar[1]{\Bigl(#1\Bigr)}
\newcommand\biggpar[1]{\biggl(#1\biggr)}
\newcommand\lrpar[1]{\left(#1\right)}

\newcommand\Bigabs[1]{\Bigl|#1\Bigr|}

\def\rompar(#1){\textup(#1\textup)}    

\newcommand\parfrac[2]{\Bigpar{\frac{#1}{#2}}}
\newcommand\parfracc[2]{\biggpar{\frac{#1}{#2}}}
\newcommand\parfracx[2]{\lrpar{\frac{#1}{#2}}}

\def\xexp(#1){e^{#1}}
\newcommand\ceil[1]{\lceil#1\rceil}

\newcommand\ntoo{\ensuremath{{n\to\infty}}}

\newcommand\norm[1]{\|#1\|}

\newcommand\Bignorm[1]{\Bigl\|#1\Bigr\|}
\newcommand\downto{\searrow}

\newcommand\ie{i.e.\spacefactor=1000}
\newcommand\eg{e.g.\spacefactor=1000}

\newcommand\cf{cf.\spacefactor=1000}

\newcommand\whp{{w.h.p.\spacefactor=1000}}

\newcommand{\tend}{\longrightarrow}
\newcommand\dto{\overset{\mathrm{d}}{\tend}}
\newcommand\pto{\overset{\mathrm{p}}{\tend}}

\newcounter{CC}

\newcommand\E{\operatorname{\mathbb E{}}}
\renewcommand\P{\operatorname{\mathbb P{}}}
\newcommand\Var{\operatorname{Var}}
\newcommand\Cov{\operatorname{Cov}}

\newcommand\Po{\operatorname{Po}}

\newcommand\Bo{\operatorname{Bo}}

\newcommand\AsN{\operatorname{AsN}}

\newcommand\ga{\alpha}

\newcommand\gD{\Delta}

\newcommand\gam{\gamma}

\newcommand\gk{\varkappa}
\newcommand\gl{\lambda}

\newcommand\gs{\sigma}
\newcommand\gss{\sigma^2}
\newcommand\eps{\varepsilon}

\newcommand\cA{\mathcal A}

\newcommand\cC{\mathcal C}

\newcommand\cG{\mathcal G}

\newcommand\tM{{\tilde M}}
\newcommand\tS{{\tilde S}}
\newcommand\tV{{\tilde V}}
\newcommand\tW{{\tilde W}}

\newcommand\ett[1]{\boldsymbol1[#1]} 

\def\[#1]{[\![#1]\!]}

\newcommand\qq{^{1/2}}
\newcommand\qqq{^{1/3}}
\newcommand\qqa{^{2/3}}
\newcommand\qqaw{^{-2/3}}
\newcommand\qqc{^{3/2}}
\newcommand\qqw{^{-1/2}}
\newcommand\qqqw{^{-1/3}}
\newcommand\qw{^{-1}}
\newcommand\qww{^{-2}}

\renewcommand{\=}{:=}

\newcommand\oi{[0,1]}

\newcommand\dd{\,\textup{d}}

\newcommand{\pgf}{probability generating function}

\newcommand\sus{\chi}
\newcommand\cc[1]{\mathcal C_{#1}}
\newcommand\cci{\cc{i}}
\newcommand\ccj{\cc{j}}
\newcommand\ccc[1]{|\cc{#1}|}
\newcommand\ccci{\ccc{i}}
\newcommand\cccj{\ccc{j}}
\newcommand{\bigrestr}[1]{\big|_{#1}}
\newcommand\sumiK{\sum_{i=1}^K}
\newcommand\sumlki{\sum_{l=1}^{k-1}}
\newcommand\sumlkii{\sum_{l=2}^{k-2}}
\newcommand\summki{\sum_{m=1}^{k-1}}
\newcommand\sumij{\sum_{i\neq j}}
\newcommand\s[1]{S_{#1}}
\newcommand\sk{\s k}
\newcommand\skl{\s {k,l}}
\newcommand\skt{\sk(t)}
\newcommand\slt{\s l(t)}
\newcommand\sklt{\skl(t)}
\newcommand\es[1]{s_{#1}}
\newcommand\esk{\es k}
\newcommand\esl{\es l}

\newcommand\eskt{\esk(t)}
\newcommand\eslt{\esl(t)}

\newcommand\ts[1]{\tS_{#1}}
\newcommand\tsk{\ts k}
\newcommand\tskt{\tsk(t)}
\newcommand\tvk{\tV_k}

\newcommand\tmk{\tM_k}
\newcommand\tmkt{\tmk(t)}
\newcommand\ttM{\Tilde{\tM}}
\newcommand\ttmk{\ttM_k}
\newcommand\ttmkt{\ttmk(t)}
\newcommand\ttW{\Tilde{\tW}}
\newcommand\gnp{\ensuremath{G(n,p)}}
\newcommand\gnm{\ensuremath{G(n,m)}}
\newcommand\gnt{\ensuremath{\cG(n,t)}}
\newcommand\et{e^{-t}}
\newcommand\intot{\int_0^t}
\newcommand\qmk{[M_k,M_k]}
\newcommand\qmkt{\qmk_t}
\newcommand\qtmk{[\tM_k,\tM_k]}
\newcommand\qtmkt{\qtmk_t}
\newcommand\qttmk{[\ttM_k,\ttM_k]}
\newcommand\qttmkt{\qttmk_t}
\newcommand\subp{_{\mathrm{p}}}
\newcommand\op{o\subp}
\newcommand\Op{O\subp}
\newcommand\simp{\sim\subp}
\newcommand\untx{\frac1{1-nt}}
\newcommand\unt{\Bigpar{\untx}}
\newcommand\untz{\bigpar{\untx}}
\newcommand\unux{\frac1{1-nu}}
\newcommand\unu{\Bigpar{\unux}}
\newcommand\ulx{\frac1{1-\gl}}
\newcommand\ul{\Bigpar{\ulx}}
\newcommand\unyx{\frac1{1-y}}
\newcommand\uny{\Bigpar{\unyx}}
\newcommand\bax[1]{B_{#1}}
\newcommand\ba{\bax{\gl}}
\newcommand\hba{\widehat B_{\gl}}
\newcommand\hbnt{\widehat B_{nt}}
\newcommand\psitl{\psi(t;\gl)}
\newcommand\TT{T(\gl e^{-\gl}e^t)}
\newcommand\xl{_{L^1}}
\newcommand\OL{O_{L^1}}
\newcommand\OLii{O_{L^2}}
\newcommand\OLp{O_{L^p}}
\newcommand\olp{o_{L^p}}
\newcommand\ol{o_{L^1}}
\newcommand\yx[1]{Y_{#1}}
\newcommand\xkt{\yx k(t)}
\newcommand\xlt{\yx l(t)}
\newcommand\mn{M^{(n)}}
\newcommand\tnx{t_n(x)}
\newcommand\pp[1]{p_{#1}}
\newcommand\ppk{\pp k}
\newcommand\pq[1]{q_{#1}}
\newcommand\pqk{\pq k}
\newcommand\pr[1]{r_{#1}}
\newcommand\prm{\pr m}
\newcommand\px[1]{\tilde P_{#1}}
\newcommand\pxk{\px k}
\newcommand\py[1]{\tilde Q_{#1}}
\newcommand\pyk{\py k}
\newcommand\qm{\hat P}
\newcommand\qmkk{\qm_{k}}
\newcommand\qmkl{\qm_{k,l}}
\newcommand\pc[1]{\qm_{#1}}

\newcommand\qc[1]{Q_{#1}}

\newcommand\pmkj{\bar P_{k,j}}
\newcommand\pdx{p^*}
\newcommand\pd[1]{\pdx_{#1}}
\newcommand\pdk{\pd k}
\newcommand\pdl{\pd l}
\newcommand\hpkl{\hpxx{k}{l}}
\newcommand\hpxx[2]{P_{#1,#2}}
\newcommand\skkm{S_{k_1,\dots,k_m}}
\newcommand\aaa{|A|}
\newcommand\cx{c}
\newcommand\ckl{\cx_{k,l}}
\newcommand\nap{(n-\ceil{\ga n},p)}
\newcommand\sknp{\sk(n,p)}
\newcommand\sknap{\sk\nap}
\newcommand\nna{n-\ceil{\ga n}}
\newcommand\ct{C^\tau}
\newcommand\Xitau{\Xi^{(\tau)}}
\newcommand\Latau{\Lambda^{(\tau)}}

\newcommand\CS{Cauchy--Schwarz}
\newcommand\CSineq{\CS{} inequality}

\newcommand{\Takacs}{Tak\'acs}
\newcommand\ER{Erd\H os--R\'enyi}

\newcommand{\maple}{\texttt{Maple}}

\begin{document}
\title
{Susceptibility in subcritical random graphs}

\date{31 May, 2008} 

\author{Svante Janson}
\address{Department of Mathematics, Uppsala University, PO Box 480,
SE-751~06 Uppsala, Sweden}
\email{svante.janson@math.uu.se}
\urladdr{http://www.math.uu.se/\~{}svante/}

\author{Malwina J. Luczak}
\address{Department of Mathematics, London School of Economics,
  Houghton Street, London WC2A 2AE, United Kingdom}
\email{malwina@planck.lse.ac.uk}
\urladdr{http://www.lse.ac.uk/people/m.j.luczak@lse.ac.uk/}


\begin{abstract} 
We study the evolution of the susceptibility in the subcritical random
graph $G(n,p)$ as $n$ tends to infinity. We obtain precise asymptotics
of its expectation and variance, and show it obeys a law of large
numbers. We also prove that the scaled fluctuations of the susceptibility
around its deterministic limit converge to a Gaussian law. We further
extend our results to higher moments of the component size of a
random vertex, and prove that they are jointly asymptotically normal.  
\end{abstract}

\maketitle

\section{Introduction}\label{S:intro}

The \emph{susceptibility} $\sus(G)$ of a graph $G$ (deterministic or
random)
is defined as the mean size of the component containing a random
vertex. (As is well known, for random graphs of the random-cluster model, 
this, or rather its expectation, corresponds to the magnetic
susceptibility in Ising and Potts models.) 
If $G$ has $n$ vertices and 
components $\cc1,\dots,\cc K$, where $K$ is the number of
components, then thus
\begin{equation}
  \sus(G)=\sumiK \frac{\ccci}{n}\ccci
=\frac1n\sumiK \ccci^2.
\end{equation}
We define, for integers $k\ge1$,
\begin{equation}\label{skg}
\sk(G)\=\sumiK \ccci^k.
\end{equation}
Thus $\sus(G)=n\qw\s2(G)$, and similarly $n\qw\s{m+1}$ is the $m$th
moment of the size of the component containing a random vertex.
(Note that by choosing a uniform random vertex, we bias the components
by their sizes. The mean size of a uniformly chosen random component
is $n/K$, which is different and which will not be treated here.)

The purpose of this paper is to study $\sus(\gnp)$, or equivalently
$\s2(\gnp)$ for the standard \ER{}
random graph $\gnp$ with $n$ vertices where 
each possible edge appears with probability $p$, independently of all
other edges; we will also give extensions to $\sk(\gnp)$ for larger
$k$.

We consider asymptotics as $\ntoo$, with $p=p(n)$ a function of $n$.
(All unspecified limits are as \ntoo.) 

It is well-known, see \eg{} 
\citet{Bollobas} and \citet{JLR}, that if $np$ is a little larger than
1, $np-1\gg n\qqqw$ to be precise, then \gnp{} has \whp{} a giant component
which is much larger than the others (the \emph{supercritical case}).
It is easily seen that then the giant component will dominate all
other terms in the sum \eqref{skg}; 
hence, if the largest component is $\cC_1$,
then $S_k(\gnp)=(1+\op(1))\ccc1^k$ and
$\sus(\gnp)=(1+\op(1))\ccc1^2/n$.
See \refApp{Asuper} for a more precise statement (and proof).

Similarly, if $np=1+O(n\qqqw)$ (the \emph{critical case}), then there
are several components of the order $n^{2/3}$; in this case $S_k$ will
be of order $n^{2k/3}$, and thus $\sus$ of order $n\qqq$, and it
follows from \citet{Aldous}
that these quantities, properly
normalized, converge in distribution to some random variables but not
to constants. 
See \refApp{Acritical} for details.

In this paper we therefore concentrate on the case $np<1$, and in
particular $1-np\gg n\qqqw$ (the \emph{subcritical case}).
We will prove the following results for $\sus(\gnp)$, together with
similar results for $\sk(\gnp)$ stated later.

We use $\Op$ and $\op$ in the standard sense,
see \eg{} \cite[pp.\ 10--11]{JLR}, and write $X_n\simp a_n$ for
$X_n=a_n+\op(a_n)$ or, equivalently, $X_n/a_n\pto1$.
We will also write $X_n=\OLp(a_n)$ if
$\norm{X_n}_{L^p}\=(\E|X_n|^p)^{1/p}=O(a_n)$, and, 
similarly, $X_n=\olp(a_n)$ if $\norm{X_n}_{L^p}=o(a_n)$.
(Here, $X_n$ and $a_n$ are sequences of random variables and positive
numbers.)

\begin{theorem}  \label{T1}
Uniformly, for all $n\ge1$ and $0\le p<n\qw$,
  \begin{align}
\E\sus(\gnp)&=\frac1{1-np}\Bigpar{1+O\Bigpar{\frac1{n(1-np)^3}}},
\label{t1e}
\\
\Var\sus(\gnp)&=O\Bigpar{\frac1{n(1-np)^5}},
\label{t1var}
\intertext{and}
\sus(\gnp)&=\frac1{1-np}\Bigpar{1+\Op\Bigpar{\bigpar{n(1-np)^3}\qqw}}.
\label{t1sus}
  \end{align}
In particular, if\/ $1-np\gg n\qqqw$, then $\sus(\gnp)\simp 1/(1-np)$.
\end{theorem}

One way to handle to explosion at $p=1/n$ is to consider
$1/\E\sus$ or $1/\sus$. 
In this form we can obtain uniform estimates for all $p$.

\begin{corollary}\label{C1}
Uniformly, for all $n\ge1$ and $0\le p\le 1$,
  \begin{align}
\frac1{\E\sus(\gnp)}&=(1-np)_++O\bigpar{n\qqqw},
\label{c1e}
\\
\frac1{\sus(\gnp)}&=(1-np)_++\Op\bigpar{n\qqqw}.
\label{c1sus}
  \end{align}
\end{corollary}

The last statement of \refT{T1} can be sharpened to asymptotic normality.
We will also find the variance more precisely. 
We write $X_n\sim\AsN(\mu_n,\gss_n)$ if $(X_n)$ is a sequence of random
variables and $\mu_n$ and $\gs_n>0$ are real numbers such that
$(X_n-\mu_n)/\gs_n\dto N(0,1)$.

\begin{theorem} \label{T2}
If $p=p(n)<n\qw$ and further $1-np\gg n\qqqw$, then
  \begin{align*}
\sus(\gnp)\sim\AsN\Bigpar{\frac1{1-np},\;\frac{2p}{(1-np)^5}}
  \end{align*}
and 
$\Var\sus(\gnp)\sim 2p/(1-np)^5$.
\end{theorem}

It follows easily from $\sus(\gnp)>0$ that 
the asymptotic normality in \refT{T2} cannot hold for
$1-np=O(n\qqqw)$. 

The proof of \refT{T1} (given in Sections \ref{Sexp}--\ref{Svar})
is fairly simple and is
based on studying how $\s k$ evolves for the 
\ER{} random
graph process \gnt{} (defined in  \refS{Sprel}). Heuristically, it is
easy to see that (ignoring the difference between a random variable
and its mean), $\sk$ ought to be an approximative solution to the 
differential equation $f'(t)=f^2(t)$, which (with the initial value
$f(0)=n$) is solved by $f(t)=n/(1-nt)$. We make this precise and
rigorous below.
This simple idea has presumably been noticed by several people, and at
least the leading terms in \eqref{t1e} and \eqref{t1sus} are more or
less known folk theorems.
However, we do not know of any rigorous treatments, except
\cite{SW} which uses the susceptibility to study a class of more complicated
random graph process. Their processes include the \ER{} process
studied here, so their results include the leading term asymptotics in
\eqref{t1e} and \eqref{t1sus} in the case where $p\le (1-\eps)/n$ for some
constant $\eps>0$. Their analysis involves branching processes
approximation, as well as differential equations, and seems contingent
on the fact that the component distribution (excluding the giant in
the supercritical case) has exponentially decaying tails.

The proof of \refT{T2} 
is more involved; the asymptotic normality is based on using a
martingale central limit theorem for a suitable modification of the
process $S_k(\gnt)$
(\refS{Snormal}), 
while the variance is estimated directly (\refS{Svar2}).

In \refS{Sborel}, the asymptotic results for $\sk$ are interpreted 
using the Borel distribution and its moments. 

\begin{remark}
  It is seen from \refT{T1} that the susceptibility blows up at
  $p=1/n$, which of course is another sign of the phase transition
  there, with the emergence of a giant component.
In fact, our results give a new proof that there is no giant
  component for smaller $p$. In the opposite direction, 
the explosion of the susceptibility at (or close to) $p=1/n$ shows
  that there are large components at that stage;
it is tempting  to conclude that a giant component emerges around this
instance (as we know by other argumants), but a formal proof based on
  this seems to require some additional work. See \citet{SW} where
  this type of arguments is used for a class of more complicated
  random graph processes. 
\end{remark}

\begin{remark}
An alternative approach to at least some of our results is to use
the standard branching process approximation of the neighbourhood
exploration process; this will be treated elsewhere.  
\end{remark}

\begin{remark}
In this paper we study the random graph \gnp. 
Most or all of our results transfer easily to the random graph \gnm{}
with a fixed number  of edges by monotonicity (\refL{L0}) and the
standard device of coupling \gnm{} with \gnp{} for a suitable $p$ such
that the expected number of edges is slightly smaller or larger
than $m$. We leave the details to the reader.
\end{remark}

\begin{ack}
This work was initiated during the programme
``Combinatorics and Statistical Mechanics'' at the Isaac Newton
Institute, Cambridge, 2008, where SJ was supported by a Microsoft fellowship.
\end{ack}

\section{Preliminaries}\label{Sprel}

We first note a simple monotonicity.
\begin{lemma}
  \label{L0}
If $H$ is a subgraph of $G$, then $\sk(H)\le\sk(G)$ for every $k\ge1$.  
\end{lemma}
\begin{proof}
  It suffices to consider the case when $G$ is obtained from $H$ by
  either adding a single edge or adding a single vertx (and no edges);
  both cases are immediate.
\end{proof}

\smallskip
The random graph process \gnt{} starts at $t=0$ with
$n$ vertices and no edges, and where edges are added randomly and
independently to every possible pair of vertices with rate 1, \ie, the
time edge $ij$ is added has an exponential distribution with mean 1.
Hence, at a given time $t$, each possible edge is present with
probability $1-\et$, so \gnt{} is a random graph $G(n,1-\et)$. We are
interested in the subcritical case where $t<1/n$; then the difference
between $1-\et$ and $t$ is 
$O(t^2)=O(n\qww)$ which is
negligible, and we can see $\gnt$ as a
convenient version of $G(n,t)$. More precisely, \gnp{} can be
obtained as $\cG(n,-\log(1-p))$;
this slight reparametrization is annoying but harmless, and it will be
convenient in the proofs below.

We write $\skt$ for $\sk(\gnt)$.
(These and other quantities introduced below depend on $n$, but we
choose not to show this explicitly in the notation.)

\smallskip

We further define, for a graph $G$ with components $\cci$ and $k,l\ge1$,
\begin{equation}\label{skl}
\skl(G)\=\sumij \ccci^k\cccj^l
=\sk(G)\s l(G)-\s{k+l}(G).
\end{equation}
We write $\sklt$ for $\skl(\gnt)$.

\section{The expectation}\label{Sexp}

We may and will assume that the edges are added to $\gnt$ at distinct times.
If a new edge joins two different components $\cci$ and $\ccj$ in
\gnt, then $\skt$ increases by a jump
\begin{equation}
  \label{dsk}
\gD\skt
=
\bigpar{\ccci+\cccj}^k-\ccci^k-\cccj^k
=\sumlki\binom kl\ccci^{l}\cccj^{k-l}.
\end{equation}
For each unordered pair $(i,j)$, the intensity of such jumps equals the
number of possible edges joining the two components, \ie~$\ccci\cccj$.
We consider ordered pairs of components and therefore divide this by
2, and summing over all pairs we find that the drift of $\skt$ is
\begin{equation}
  \label{v}
V_k(t)
\=
\sumij\frac12\ccci\cccj\sumlki\binom kl\ccci^{l}\cccj^{k-l}
=
\sumlki\frac12\binom kl\s{l+1,k+1-l}(t);
\end{equation}
in other words,
noting that $S_k(0)=n$,
\begin{equation}
  M_k(t)\=\skt-n-\intot V_k(u)\dd u
\end{equation}
is a martingale on $[0,\infty)$ with $M_k(0)=0$.
(Note that $M_k(t)$ 
is bounded for each fixed $n$ and $t$ in a finite interval $[0,T]$;
hence, there are no problems with integrability of this martingale.
The same holds for all similar martingales below.)

We define $\esk(t)\=\E\skt$, 
noting that $\esk(0)=n$, and conclude from the martingale property that
$\E M_k(t)=\E M_k(0)=0$ and thus
\begin{equation}\label{eskt}
  \eskt=\E S_k(t)=n+\intot\E V_k(u)\dd u.
\end{equation}

In order to use this, we need information on $\E\sklt$.

\begin{lemma}
  \label{L1}
For all $k,l\ge1$:
  \begin{romenumerate}
\item
$\E\sklt \le\eskt\eslt$,
\item
$\E\sklt \ge\eskt\eslt-\es{k+l}(t)$.	
  \end{romenumerate}
\end{lemma}
\begin{proof}
\pfitem{i}
Let $\cA_n$  be the set of all non-empty subsets of $[n]$. If
$A\in\cA_n$, let $I_A(t)\=\ett{A \text{ is a component of }\gnt}$.
Thus,
\begin{equation*}
  \skt=\sum_{A\in\cA_n}|A|^k I_A(t)
\end{equation*}
and,
since $I_AI_B=0$ if $A\cap B\neq\emptyset$ but $A\neq B$,
\begin{equation}\label{a1}
  \sklt=\sum_{A\neq B}|A|^k|B|^l I_A(t)I_B(t)
=\sum_{A\in\cA_n}|A|^k I_A(t) \sum_{B\subseteq [n]\setminus A}|B|^l I_B(t).
\end{equation}
Conditioned on $I_A(t)=1$, the
conditional distribution of the restriction of \gnt{} to 
$[n]\setminus A$ is a random graph with the same distribution as
$\cG(n-|A|,t)$, apart from a relabelling of the vertices. Hence, using
also \refL{L0},
\begin{equation*}
\E\Bigl(\sum_{B\subseteq [n]\setminus A}|B|^l I_B(t)\Bigm|I_A(t)=1\Bigr)
=\E S_l(\cG(n-|A|,t))
\le \E S_l(\cG(n,t))
=\esl(t).  
\end{equation*}
Consequently, taking the expectation in \eqref{a1} yields
\begin{equation*}
\E  \sklt
\le
\E\sum_{A\in\cA_n}|A|^k I_A(t) \eslt
=\eskt\eslt.
\end{equation*}

\pfitem{ii}
By \eqref{skl},
\begin{equation*}
\E\skl(t)
=
\E\bigpar{\sk(t)\s l(t)}-\es{k+l}(t),
\end{equation*}
and it remains to show that 
$\E\bigpar{\sk(t)\s l(t)}\ge\eskt\eslt$, \ie, that $\skt$ and $\slt$
are positively correlated. This follows by Harris' inequality
(a special case of the FKG inequality), since
$\skt$ and $\slt$ are (by \refL{L0})
increasing functions of the edge indicators of
$\gnt$, and these are independent.  
\end{proof}

We use this first to find an upper bound for $\eskt$.
Combining \eqref{eskt}, \eqref{v} and \refL{L1}(i), we find
\begin{equation}\label{esk'}
  \esk'(t)=\E V_k(t) \le \sumlki\frac12\binom kl \es{l+1}(t)\es{k-l+1}(t).
\end{equation}
The first cases are
\begin{align}
  \es2'(t) & \le \es2(t)^2,
\label{es2'}
\\
  \es3'(t) & \le 3\es2(t)\es3(t),
\label{es3'}
\\
  \es4'(t) & \le 4\es2(t)\es4(t)+3\es3(t)^2.
\label{es4'}
\end{align}

Integrating \eqref{es2'}, with the initial value $\es2(0)=n$, we find,
\eg{} via $(1/\es2(t))'\ge-1$ and 
thus $1/\es2(t)\ge 1/n-t$,
\begin{equation}
  \label{es2}
\es2(t)\le \frac{n}{1-nt},
\qquad 0\le t<1/n.
\end{equation}
Next, \eqref{es3'} and \eqref{es2} yield
$\bigpar{(1-nt)^3\es3(t)}'\le0$ and thus, since $\es3(0)=n$,
\begin{equation}
  \label{es3}
\es3(t)\le \frac{n}{(1-nt)^3},
\qquad 0\le t<1/n.
\end{equation}
We can continue recursively and obtain the following bounds.
\begin{lemma}\label{Lesk}
  For every $k\ge2$, there exists a constant $C_k$ such that, for all
  $n$,
\begin{equation*}
\E\sk(t)=
\esk(t)\le C_k\frac{n}{(1-nt)^{2k-3}},
\qquad 0\le t<1/n.
\end{equation*}
\end{lemma}
\begin{proof}
  We have proven this for $k=2$ and $3$.
For $k\ge4$ we use induction and assume that the lemma holds for
smaller values of $k$; then \eqref{esk'} yields, for some constants
$C_k'$ and $C_k''$, 
taking the terms $l=1$ and $l=k-1$ separately and using \eqref{es2},
\begin{equation*}
  \begin{split}
\esk'(t)
&\le k\es2(t)\esk(t)
+
\sum_{l=2}^{k-2}
C_k'
\frac{C_{l+1}n}{(1-nt)^{2l-1}}	
\frac{C_{k-l+1}n}{(1-nt)^{2k-2l-1}}	
\\&
\le 
\frac{kn}{1-nt}\esk(t)
+
\frac{C''_{k}n^2}{(1-nt)^{2k-2}}.	
  \end{split}
\end{equation*}
Hence, 
$\bigpar{(1-nt)^k\esk(t)}'\le C''_kn^2(1-nt)^{-(k-2)}$ and thus
\begin{equation*}
  \begin{split}
(1-nt)^k\esk(t)
&\le
n+\intot \frac{C''_{k}n^2}{(1-nu)^{k-2}}\dd u
\le
n+\frac{C''_kn}{(k-3)(1-nt)^{k-3}}
\\&
\le \frac{C_kn}{(1-nt)^{k-3}}.	
  \end{split}
\qedhere
\end{equation*}
\end{proof}

We write the estimate in \refL{Lesk} as $\esk(t)=O(n(1-nt)^{3-2k})$
where, as in all similar estimates below, the implicit constant may
depend on $k$ (and later sometimes $l$), but not on $n$ or $t$ (in the
given range $0\le t<1/n$).

We can now use this upper bound in a more or less repetition of the
same argument 
to obtain more precise estimates.
By Lemmas \refand{L1}{Lesk}, for $0\le t<1/n$,
\begin{equation*}
\E\sklt =\eskt\eslt+O(\es{k+l}(t))
=\eskt\eslt+O\biggpar{\frac{n}{(1-nt)^{2k+2l-3}}}.
\end{equation*}
Hence, \eqref{esk'} and \eqref{v} yield
\begin{equation}
  \label{esk'b}
\esk'(t)
=\E V_k(t)
=
\sumlki\frac12\binom kl \es{l+1}(t)\es{k-l+1}(t)
+O\biggpar{\frac n{(1-nt)^{2k+1}}}.
\end{equation}
The first cases are
\begin{align}
  \es2'(t) & = \es2(t)^2+O\bigpar{n(1-nt)^{-5}},
\label{es2'b}
\\
  \es3'(t) & = 3\es2(t)\es3(t)+O\bigpar{n(1-nt)^{-7}},
\label{es3'b}
\\
  \es4'(t) & = 4\es2(t)\es4(t)+3\es3(t)^2+O\bigpar{n(1-nt)^{-9}}.
\label{es4'b}
\end{align}

We first treat $\es2(t)$.
\begin{theorem}\label{Tes2}
  \begin{equation*}
  \E\s2(t)=\es2(t)=\frac n{1-nt}\Bigpar{1+O\Bigpar{\frac{nt}{n(1-nt)^3}}},
\qquad 0\le t<1/n.	
  \end{equation*}
\end{theorem}

\begin{proof}
  Let $T\=\inf\set{t:(1-nt)\es2(t)= n/2}$. 
Since $f(t)\=(1-nt)\es2(t)$ is continuous with $f(0)=n$ and
$f(1/n)=0$,
then
$0<T<1/n$ and for
  $0\le t\le T$ we have $\es2(t)\ge\tfrac12n/(1-nt)$ and thus, by
  \eqref{es2'b},
\begin{equation*}
\biggpar{\frac1{\es2(t)}}'
=-1+O\biggpar{\frac n{(1-nt)^5\es2(t)^2}}	
=-1+O\biggpar{\frac1{n(1-nt)^3}}.
\end{equation*}
This implies, recalling $\es2(0)=n$ and noting that
$\intot(1-nu)^{-3}\dd u=O(t/(1-nt)^2)$
(which is, like similar integrals below, perhaps simplest seen by
considering the cases $nt\le1/2$ and $nt\ge1/2$ separately),
\begin{equation}\label{fdag}
  \begin{split}
{\frac1{\es2(t)}}
&=\frac1n+\intot \Bigpar{\frac1{\es2(u)}}'\dd u
=\frac1n-t+O\Bigpar{\frac {t}{n(1-nt)^2}}	
\\&
=\frac{1-nt}{n}\Bigpar{1+O\Bigpar{\frac{nt}{n(1-nt)^3}}}
.
  \end{split}
\end{equation}
Taking here $t=T$, we find $1=O\bigpar{1/(n(1-nT)^3)}$, and thus
$n(1-nT)^3=O(1)$ or
$1-nT=O(n\qqqw)$.
Choosing $A$ large enough, we see
that if $1-nt\ge An\qqqw$, then $t\le T$, and further the $O$ term in 
\eqref{fdag} is, in absolute value, less than $1/2$. Thus \eqref{fdag} yields
the result for $1-nt\ge An\qqqw$. The result for $1-nt<An\qqqw$
follows trivially from the bound \eqref{es2}.
\end{proof}

\refT{Tes2} proves \eqref{t1e} by the change of variable
$t=-\log(1-p)=p+O(p^2)$ 
as discussed in \refS{Sprel}, noting that the result is utterly
trivial for $1-np=O(n\qw)$.

We continue with higher $k$.

\begin{theorem}\label{Tesk}
The following holds for $0\le t<1/n$.
  \begin{align*}
  \E\s3(t)&=\es3(t)=\frac n{(1-nt)^3}\Bigpar{1+O\Bigpar{\frac{nt}{n(1-nt)^3}}},
\\
  \E\s4(t)&=\es4(t)
=\frac {n\bigpar{3-2(1-nt)}}{(1-nt)^5}\Bigpar{1+O\Bigpar{\frac{nt}{n(1-nt)^3}}}.
  \end{align*}
More generally, for every $k\ge2$ there exists a polynomial $\pp k$ of
degree $2k-3$ such that
\begin{equation}\label{tesk}
  \begin{split}
  \E\sk(t)=\esk(t)
&=n\ppk\Bigpar{\frac1{1-nt}}+O\Bigpar{\frac{nt}{(1-nt)^{2k}}}	
\\
&=n\ppk\Bigpar{\frac1{1-nt}}\Bigpar{1+O\Bigpar{\frac{nt}{n(1-nt)^3}}}.
  \end{split}
\end{equation}
We have $\pp2(x)=x$, $\pp3(x)=x^3$, $\pp4(x)=3x^5-2x^4$. In general,
for $k\ge3$, 
$\ppk(x)=x^k\pqk(x)$ for a polynomial $\pqk(x)$ of degree $k-3$ that is
recursively defined by $\pqk(1)=1$ and
\begin{equation}\label{qk1}
  \pqk'(x)=\frac12\sumlkii\binom kl \pq{l+1}(x)\pq{k-l+1}(x),
\qquad k\ge3.
\end{equation}
Equivalently,
$\ppk(1)=1$ and
\begin{equation}\label{pk1}
  \ppk'(x)=\frac1{2x^2}\sumlki\binom kl \pp{l+1}(x)\pp{k-l+1}(x),
\qquad k\ge2.
\end{equation}
\end{theorem}

A probabilistic interpretation of $\ppk(x)$ and a simpler recursion
formula are given in \refS{Sborel}.
The polynomials $\ppk$ for small $k$ are given in \refTab{Tabpk}.

\begin{table}[ht]
\begin{align*}
\pp2(x)&=x,
\\
\pp3(x)&={x}^{3}, \\
\pp4(x)&=3\,{x}^{5}-2\,{x}^{4}, \\
\pp5(x)&=15\,{x}^{7}-20\,{x}^{6}+6\,{x}^{5}, \\
\pp6(x)&=105\,{x}^{9}-210\,{x}^{8}+130\,{x}^{7}-24\,{x}^{6}, \\
\pp7(x)&=945\,{x}^{11}-2520\,{x}^{10}+2380\,{x}^{9}-924\,{x}^{8}+120\,{x}^{7}
, \\
\pp8(x)&=10395\,{x}^{13}-34650\,{x}^{12}+44100\,{x}^{11}-26432\,{x}^{10}+7308\,
{x}^{9}-720\,{x}^{8}.
\end{align*}
\caption{The polynomials $\ppk(x)$ for $k\le8$.} \label{Tabpk}
\end{table}

\begin{proof}
We have shown the result for $k=2$, with $p_2(x)=x$ which satisfies
\eqref{pk1}.
For larger $k$, we use induction
and assume that \eqref{tesk} is true for smaller values of $k$.
Then, by   \eqref{esk'b}, taking the terms $l=1$ and $l=k-1$
separately,
and \eqref{qk1},
\begin{equation*}
  \begin{split}
\esk'(t)
&=
k\es2(t)\esk(t)+
\sum_{l=2}^{k-2}\frac12\binom kl \es{l+1}(t)\es{k-l+1}(t)
+O\Bigpar{\frac n{(1-nt)^{2k+1}}}
\\&
=\frac {kn}{1-nt}\esk(t)+
n^2 \sum_{l=2}^{k-2}\frac12\binom kl \pp{l+1}\unt\pp{k-l+1}\unt
\\&
\hskip 18em
{}+O\Bigpar{\frac n{(1-nt)^{2k+1}}}
\\
&=\frac {kn}{1-nt}\esk(t)+
\frac{n^2}{(1-nt)^{k+2}}\pqk'\unt+O\Bigpar{\frac n{(1-nt)^{2k+1}}}.
  \end{split}
\end{equation*}
Thus,
\begin{equation*}
  \begin{split}
\bigpar{(1-nt)^k\esk(t)}'
&=
\frac{n^2}{(1-nt)^{2}}\pqk'\unt+O\Bigpar{\frac n{(1-nt)^{k+1}}}.
\\&
=n\frac{\dd}{\dd t}
\pqk\unt+O\Bigpar{\frac n{(1-nt)^{k+1}}}.
  \end{split}
\end{equation*}
The result follows by integration, recalling that $\esk(0)=n$. For the
second form in \eqref{tesk}, with the error term written
multiplicatively, we note also that it follows from the recursion
\eqref{qk1} that $\ppk$ has degree $2k-3$ with a positive leading term;
since further $\pqk$ and $\ppk$ are non-decreasing on $[1,\infty)$, for
  example by \eqref{qk1} again, and thus strictly positive there, it
  follows that
$\ppk(x)\asymp x^{2k-3}$ for $x\ge1$. 
\end{proof}

\section{The variance}\label{Svar}

\begin{theorem}
  \label{Tvar}
For every $k\ge2$,
$\Var(\skt) \le \es{2k}(t)$.
Hence, 
\begin{equation*}
 \Var(\skt) = O\bigpar{n(1-nt)^{-(4k-3)}},
\qquad 0\le t<1/n.
\end{equation*}
\end{theorem}

\begin{proof}
  By \eqref{skl} and \refL{L1}(i),
  \begin{equation*}
\E(\skt^2)
=\E\s{k,k}(t)+\E\s{2k}(t)
\le(\E\skt)^2+\es{2k}(t).
  \end{equation*}
The final estimate follows by \refL{Lesk}.
\end{proof}

A more precise result will be given in \refS{Svar2}. This will show
that the bound in \refT{Tvar} is of the right order as
long as $1-nt\gg n\qqqw$.

\begin{corollary}
  \label{Cvar}
If $1-nt\gg n\qqqw$, then
$\skt\simp n\ppk\unt$
for every $k\ge2$.
\end{corollary}

\begin{proof}
  By \refT{Tesk}, $\E\skt\sim n\ppk\untz$. 
Further, Theorems \refand{Tvar}{Tesk} show that
\begin{equation*}
\frac{\Var(\skt)}{(\E\skt)^2} =
O\parfracc{n(1-nt)^{3-4k}}{n^2(1-nt)^{6-4k}}  
=O\parfracc{1}{n(1-nt)^{3}}  =o(1),
\end{equation*}
and the result follows by Chebyshev's inequality.
\end{proof}

\begin{proof}[Proof of \refT{T1}]
As remarked above, \eqref{t1e} follows from \refT{Tes2}.
Similarly, the case $k\ge2$ of \refT{Tvar} yields \eqref{t1var}.
Together, these estimates yield \eqref{t1sus} for $1-np\ge n\qqqw$;
in the remaining case $0<1-np<n\qqqw$, \eqref{t1sus} follows trivially
from the estimate $\E\sus(\gnp)\le1/(1-np)$, which follows from
\refL{Lesk} provided $1-np\ge 1/n$, and otherwise from the trivial
$\sus(\gnp)\le n$.   
\end{proof}

\begin{proof}[Proof of \refC{C1}] 
  Let $A>0$ be so large that the $O$ term in \eqref{t1e} is $\le 1/2$
  for $1-np\ge A n\qqqw$. Then 
\eqref{t1e} yields,  for $np\le 1-An\qqqw$,
\begin{equation*}
  \frac1{\E\sus(\gnp)}=(1-np)\Bigpar{1+O\Bigpar{\frac1{n(1-np)^3}}}
=1-np+O\bigpar{n\qqqw},
\end{equation*}
which shows \eqref{c1e} for these $p$.
In particular, for $np=1-An\qqqw$ we find
$1/\E\sus(\gnp)=O(n\qqqw)$. This, and thus \eqref{c1e}, then holds for
  all larger $p$ too by monotonicity (\refL{L0}).

The proof of \eqref{c1sus} is similar, using \eqref{t1sus}.
\end{proof}

\section{Asymptotic normality}\label{Snormal}

The quadratic variation of the martingale $M_k(t)$ is 
\begin{equation*}
 \qmkt \=\sum_{0<u\le t}\gD M_k(u)^2
=\sum_{0<u\le t}\gD S_k(u)^2,
\end{equation*}
where $\gD X(s)\=X(s)-X(s-)$ denotes the jump (if any) of a process $X$ at $s$.
(This formula holds because $M_k$ is a martingale with paths of finite
variation and $M_k(0)=0$; see 
\eg{} \cite{JS} for a definition for general (semi)martingales.)
Using \eqref{dsk}, we find, in analogy with \eqref{v}, that $\qmkt$
has drift
\begin{equation}
  \label{w}
  \begin{split}
W_k(t)
&\=
\sumij\frac12\ccci\cccj\lrpar{\sumlki\binom kl\ccci^{l}\cccj^{k-l}}^2
\\&
\phantom:
=
\sumlki\summki\frac12\binom kl\binom km\s{l+m+1,\;2k+1-l-m}(t);	
  \end{split}
\end{equation}
\ie, $\qmkt-\intot W_k(u)\dd u$ is a martingale.

It turns out to be advantageous to work with a slightly different
martingale.
In order to cancel some terms later on, we multiply $\skt$ by
$(1-nt)^k$
(\cf{} the proof of \refT{Tesk} where we did the same with the
expectation in order to simplify the differential equation); we thus define
\begin{equation}\label{tsk}
  \tskt\=(1-nt)^k\skt,
\end{equation}
which (by a simple instance of Ito's formula) has the drift
\begin{equation}  \label{tv}
  \tvk(t)\=(1-nt)^k V_k(t)-kn(1-nt)^{k-1} S_k(t).
\end{equation}
Thus,
\begin{equation}\label{tmkt}
  \tmkt\=\tskt-n-\intot\tvk(u)\dd u
\end{equation}
is a martingale with $\tmk(0)=0$.
The quadratic variation is
\begin{equation*}
 \qtmkt \=\sum_{0<u\le t}\gD \tM_k(u)^2
=\sum_{0<u\le t}\gD \tS_k(u)^2
=\sum_{0<u\le t}(1-nu)^{2k}\gD S_k(u)^2.
\end{equation*}
This has drift 
\begin{equation}
  \label{tw}
\tW_k(t)\=(1-nt)^{2k} W_k(t),
\end{equation}
and thus
\begin{equation}\label{ttmkt}
\ttmk(t)\=  \qtmkt-\intot \tW_k(u)\dd u
\end{equation}
is another martingale with $\ttmk(0)=0$.

We repeat the argument and find that $\ttM_k$ has quadratic variation
\begin{equation*}
  \begin{split}
 \qttmkt 
&\=\sum_{0<u\le t}\gD \ttM_k(u)^2
=\sum_{0<u\le t}(\gD \qtmk_u)^2
=\sum_{0<u\le t}\gD \tmk(u)^4
\\ & \phantom:
=\sum_{0<u\le t}(1-nu)^{4k}\gD S_k(u)^4,
  \end{split}
\end{equation*}
which has drift, in analogy with \eqref{v} and \eqref{w},
\begin{equation}
\label{ttw}
  \begin{split}
\ttW_k(t)
&\=
(1-nt)^{4k}
\sumij\frac12\ccci\cccj\lrpar{\sumlki\binom kl\ccci^{l}\cccj^{k-l}}^4
\\&
\phantom:
=(1-nt)^{4k}
\sum_{l_1,l_2,l_3,l_4=1}^{k-1}\frac12\prod_{i=1}^4\binom k{l_i}
\cdot
\s{\sum_il_i+1,\;4k+1-\sum_il_i}(t);
  \end{split}
\end{equation}
thus, $\qttmkt-\intot \ttW_k(u)\dd u$ is 
yet another martingale which starts at 0.

Assume in the remainder of the section that $1-nt\ge n\qqqw$, \ie{}
\begin{equation}\label{tt}
  0 \le t \le  n\qw-n^{-4/3}.
\end{equation}
(Although some estimates require only $0\le t< 1/n$.)
By Lemmas \ref{L1}(i) and \ref{Lesk}, for any $k,l\ge2$,
\begin{equation*}
  \E \sklt = O\parfracx{n^2}{(1-nt)^{2k+2l-6}}.
\end{equation*}
Hence, \eqref{ttw} yields
\begin{equation*}
  \E\ttW_k(t) = O\biggpar{(1-nt)^{4k}\frac{n^2}{(1-nt)^{8k-2}}}
=O\lrpar{\frac{n^2}{(1-nt)^{4k-2}}}.
\end{equation*}
Since $\Var(M(t))=\E M^2=\E[M,M]_t$ for every square integrable martingale
with $M(0)=0$,
\begin{equation}
\label{ettmkt}
  \begin{split}
\E \bigpar{\ttM_k(t)}^2
&=
\E\qttmkt
=
\E\intot \ttW_k(u)\dd u
=
O\lrpar{\intot \frac{n^2}{(1-nu)^{4k-2}}\dd u}
\\ &
=
O\parfracx{n^2t}{(1-nt)^{4k-3}}.
  \end{split}
\raisetag{\baselineskip}
\end{equation}

We define, subtracting by \eqref{tesk} an approximation to the mean,
\begin{equation}
  \label{yx}
\xkt\=\skt-n\ppk\unt.
\end{equation}

\begin{lemma}\label{Lyx}
For every $k\ge2$ and $1-nt\ge n\qqqw$,
\begin{equation*}
  \begin{split}
\xkt=\OLii\parfracc{n\qq}{(1-nt)^{2k-3/2}}.
  \end{split}
\end{equation*}
\end{lemma}

\begin{proof}
\begin{equation*}
  \begin{split}
\norm{\xkt}_{L^2}^2
={\Var\skt}	+\Bigabs{\E\skt-n\ppk\unt}^2,
  \end{split}
\end{equation*}
and the result follows by Theorems \refand{Tvar}{Tesk},
using $n(1-nt)^3\ge1$.
\end{proof}

\begin{lemma}
  \label{Lsklt}
For every $k,l\ge2$ and $1-nt\ge n\qqqw$,
\begin{equation*}
  \sklt=n^2\ppk\unt\pp l\unt +\OL\parfracx{n\qqc}{(1-nt)^{2k+2l-9/2}}.
\end{equation*}
\end{lemma}

\begin{proof}
By \eqref{skl} and \eqref{yx}, 
\begin{equation*}
  \sklt=\Bigpar{n\ppk\unt+\xkt}\Bigpar{n\pp l\unt+\xlt}-\s{k+l}(t)
\end{equation*}
and thus, using Lemmas \refand{Lyx}{Lesk} and the \CSineq{},
\begin{multline*}
\Bignorm{\sklt-n^2\ppk\unt\pp l\unt}_{L^1}
\\
=O\Bigpar{
n\qqc(1-nt)^{-2k-2l+9/2}
+
n(1-nt)^{-2k-2l+3}
},
\end{multline*}
which yields the result by our assumption 
$n(1-nt)^3\ge 1$.
\end{proof}

\begin{lemma}\label{Ltw}
For every $k\ge2$,   there exists a polynomial $\px k$ of degree
$2k-2$ given by
\begin{equation}\label{px}
  \begin{split}
 \px k(x)
&= 
x^{-2k}\sumlki\summki\frac12\binom kl\binom km\pp{l+m+1}(x)\pp{2k+1-l-m}(x)
\\
&= 
x^{2}\sumlki\summki\frac12\binom kl\binom km\pq{l+m+1}(x)\pq{2k+1-l-m}(x)
  \end{split}
\end{equation}
 such that, for $1-nt\ge n\qqqw$,
\begin{equation*}
  \tW_k(t)=n^2\px k\unt+\OL\parfracc{n\qqc}{(1-nt)^{2k-1/2}}.
\end{equation*}
\end{lemma}
\begin{proof}
  An immediate consequence of \eqref{tw}, \eqref{w} and \refL{Lsklt}.
\end{proof}

\begin{lemma}\label{Lqtm}
  \begin{thmenumerate}
	\item
For every $k\ge2$,   there exists a polynomial $\py k$ of degree
$2k-3$ given by 
\begin{equation}\label{py}
 \py k'(x)
= 
x^{-2}\pxk(x),
\end{equation}
with $\py k(1)=0$,
 such that, for $1-nt\ge n\qqqw$,
\begin{equation}\label{lqtm1} 
\qtmkt=n\py k \unt+\OL\parfracc{nt\qq}{(1-nt)^{2k-3/2}}.
\end{equation}
\item
If $n^2t\to\infty$ and $n(1-nt)^3\to\infty$, then
\begin{equation*}
\qtmkt=n\py k \unt\bigpar{1+o_{L^1}(1)}
=n\py k \unt\bigpar{1+\op(1)}.
\end{equation*}
  \end{thmenumerate}
\end{lemma}
\begin{proof}
\pfitem{i}
By \eqref{ttmkt}, \refL{Ltw} and \eqref{ettmkt},
\begin{equation*}
  \begin{split}
\qtmkt
&=\intot \tW_k(u)\dd u+\ttmkt
\\&
=\intot n^2\pxk\unu\dd u+
\OL\parfracc{n\qqc t+nt\qq}{(1-nt)^{2k-3/2}}
  \end{split}
\end{equation*}
and \eqref{lqtm1} follows, noting that
$n\qqc t\le nt\qq $.

\pfitem{ii}
By \eqref{py}, $\py k$ is increasing for $x>1$, and thus non-zero,
and it follows that $\py k\untz\asymp nt(1-nt)^{3-2k}$.
It remains only to verify that 
$nt\qq(1-nt)^{3/2}=o(n^2t(1-nt)^{3})$, which is obvious under our
conditions
if we consider
$nt\le1/2$ and $nt\ge1/2$ separately.
\end{proof}

We will use the following general result based on \cite{JS}; 
see \cite[Proposition 9.1]{SJ154} for a detailed proof.
(See also \cite{SJ79},
\cite{SJ94} and \cite{SJ196} for similar versions.)

\begin{proposition}\label{P:JS}
Assume that for each $n$, 
$\mn(x)$ is a 
martingale on $\oi$
with $\mn(0)=0$, and that 
$\gss(x)$, $x\in\oi$, 
is a  \rompar(non-random) 
continuous function 
such that
for every fixed $x\in\oi$, 
\begin{align}
&[\mn, \mn]_x\pto \gss(x)
\quad\text{as \ntoo,}
\label{js1}
\\
&\sup_n\E [\mn,\mn]_x <\infty.
\label{js2}
\end{align}
Then $\mn\dto M$ as $n\to\infty$, in 
$D\oi$, where
$M$ is a continuous $q$-dimensional Gaussian martingale with $\E M(x)=0$ and
covariances
\begin{equation*}
\E \bigpar{M(x)M(y)}=\gss (x),\quad 0\le x\le y\le1.
\end{equation*}
In particular, $\mn(1)\dto N(0,\gss(1))$.
\end{proposition}

\begin{remark}\label{R:JS}
\refP{P:JS} 
extends to vector-valued martingales;
see the versions in \cite{SJ154,SJ196}.  
\end{remark}

\begin{remark}
The versions in \cite{SJ154,SJ196} are for martingales on $[0,\infty)$; it is
  easily seen that the  versions are equivalent by stopping the
  martingales at a fixed time; moreover, by a (deterministic) change
  of time, we may replace $\oi$ by any closed or half-open interval
  $[a,b]$ or $[a,b) \subseteq[-\infty,\infty]$.

Further,  \eqref{js2} is equivalent to
$\sup_n\E |\mn(x)|^2 <\infty$, the form used in \eg{} \cite{SJ154}.
\end{remark}

\begin{lemma}\label{Ltm}
If $n^2t\to\infty$ and $n(1-nt)^3\to\infty$, then
\begin{equation*}
\tmkt\sim\AsN\Bigpar{0,n\py k \unt}.
\end{equation*}
\end{lemma}

\begin{proof}
In order to apply \refP{P:JS}, we have to change the time scale to a
fixed interval so that the quadratic variation converges.
By considering subsequences, we may assume that $nt\to a$ for some
$a\in\oi$. We then define $\mn(x)$ for $x\in\oi$ as follows.
\begin{romenumerate}
\item
If $0< a<1$, we let $\mn(x)\=(n^2t)\qqw\tmk(xt)$, and see that
\refL{Lqtm}(ii) implies \eqref{js1} with $\gss(x)=a\qw\pyk(1/(1-ax))$.
\item
If $a=0$, we define $\mn(x)$ in the same way, and find now that
\refL{Lqtm}(ii) implies \eqref{js1} with $\gss(x)=x\pyk'(1)$.
\item
If $a=1$, we let $\mn(x)\=n\qqw(1-nt)^{k-3/2}\tmk(\tnx)$, where
\begin{equation*}
  \tnx\=
  \begin{cases}
0, & x\le 1-nt, \\
\frac1n\bigpar{1-\frac{1-nt}x}, & x\ge 1-nt;
  \end{cases}
\end{equation*}
thus $1-n\tnx=\min((1-nt)/x,\,1)$. In this case \refL{Lqtm}(ii)
implies \eqref{js1} with $\gss(x)=c_kx^{2k-3}$, where $c_k>0$ is the
leading coefficient in $\pyk$.
\end{romenumerate}

In all cases, the same calculation yields also \eqref{js2}, because
the factor $1+\ol(1)$ in \refL{Lqtm} is $\OL(1)$. The result follows
from the final statement in \refP{P:JS}.
\end{proof}

Let us now consider the case $k=2$. 

\begin{theorem}\label{TS2}
If $n^2t\to\infty$ and $n(1-nt)^3\to\infty$, then
\begin{equation*}
\s2(t)\sim\AsN\lrpar{\frac{n}{1-nt},\frac{2n^2t}{(1-nt)^5}}.
\end{equation*}
\end{theorem}

\begin{proof}
By \eqref{v} and \eqref{skl}, $V_2(t)=S_{2,2}(t)=S_2(t)^2-S_4(t)$, and thus
\eqref{tv} yields
\begin{equation*}
  \begin{split}
\tV_2(t)
&=(1-nt)^2V_2(t)-2n(1-nt)S_2(t)	
\\
&=\bigpar{(1-nt)S_2(t)-n}^2-n^2-(1-nt)^2S_4(t).
  \end{split}
\end{equation*}
By Theorems \ref{Tvar} and \ref{Tes2},
\begin{equation*}
  \begin{split}
\E\bigpar{(1-nt)S_2(t)-n}^2
&=(1-nt)^2\Var(S_2(t))+
\bigpar{(1-nt)\E S_2(t)-n}^2
\\&
=O\parfrac{n}{(1-nt)^3}+O\parfrac{1}{(1-nt)^6}
=O\parfrac{n}{(1-nt)^3}.
  \end{split}
\end{equation*}
By \refL{Lesk}, $\norm{(1-nt)^2S_4(t)}\xl$ is also
estimated by $O(n(1-nt)^{-3})$.
Hence, 
\begin{equation*}
  \begin{split}
\tV_2(t)
&=-n^2+\OL\parfrac{n}{(1-nt)^3}.
  \end{split}
\end{equation*}
We now obtain from \eqref{tmkt}
\begin{equation}\label{em1}
  \ts 2(t)=\tM_2(t)+n+\intot\tV_2(u)\dd u
=\tM_2(t)+n-n^2t+\OL\parfrac{nt}{(1-nt)^2}.
\end{equation}
For $k=2$, \eqref{px} and \eqref{py} yield $\px2(x)=2x^2q_3(x)^2=2x^2$
and $\py2(x)=2(x-1)$. Hence \refL{Ltm} yields
\begin{equation}\label{em2} 
\tM_2(t)\sim\AsN\Bigpar{0,\frac{2n^2t}{1-nt}}.
\end{equation}
It is easily verified that
$\frac{nt}{(1-nt)^2}\ll\bigpar{\frac{n^2t}{1-nt}}\qq$. Hence,
  \eqref{em1} and \eqref{em2} yield
\begin{equation*}
\tS_2(t)\sim\AsN\Bigpar{n(1-nt),\frac{2n^2t}{1-nt}}.
\end{equation*}
Recalling the definition $\ts2(t)=(1-nt)^2 S_2(t)$, we obtain the
assertion. 
\end{proof}

\begin{proof}[Proof of \refT{T2}, asymptotic normality]
Immediate from \refT{TS2}  by our usual relation
  $\sus(\gnp)=n\qw\s2(-\log(1-p))$. 
\end{proof}

For $k>2$, the argument is more involved, and we will be somewhat sketchy.
We assume $1-nt\gg n\qqqw$ and consider first $k=3$.
By \eqref{v} and \eqref{skl}, 
$V_3(t)=3S_{2,3}(t)=3S_2(t)\s3(t)-3S_5(t)$, and thus
\eqref{tv} yields, using \eqref{yx}, Lemmas \refand{Lesk}{Lyx} and the \CSineq,
\begin{equation*}
  \begin{split}
\tV_3(t)
&=(1-nt)^3V_3(t)-3n(1-nt)^2S_3(t)	
\\&
=3(1-nt)^3\Bigpar{\s2(t)-\frac n{1-nt}}S_3(t)-3(1-nt)^3S_5(t)	
\\&
=3(1-nt)^3\yx2(t)S_3(t)-3(1-nt)^3S_5(t)	
\\&
=3n(1-nt)^3\pp3\unt\yx2(t)+\OL\parfrac{n}{(1-nt)^4}.
  \end{split}
\end{equation*}
Hence, by \eqref{tmkt}, recalling $\pp3(x)=x^3$,
\begin{equation}\label{em3}
  \ts 3(t)=\tM_3(t)+n+3n\intot\yx2(u)\dd u
+\OL\parfrac{nt}{(1-nt)^3},
\end{equation}
where we may ignore the $O$ term but not the integral, unlike the
corresponding expression \eqref{em1} for $k=2$.
We find from \eqref{em1} 
\begin{equation*}
  \yx2(u)
=(1-nu)\qww\bigpar{\ts2(u)-n(1-nu)}
=(1-nu)\qww\tM_2(u)+\OL\parfrac{nu}{(1-nu)^4}.
\end{equation*}
Hence, \eqref{em3} yields
\begin{equation}\label{em4}
  \ts 3(t)-n
=\tM_3(t)+3n\intot(1-nu)\qww\tM_2(u)\dd u
+\OL\parfrac{nt}{(1-nt)^3}.
\end{equation}

We applied above \refP{P:JS} to $\tM_2$, but we only used the result
\refL{Ltm} for a single $t$. Now we use the full process statement of
\refP{P:JS}, from which we conclude (after a change of variables as in
the proof of \refL{Ltm}) that 
$\intot(1-nu)\qww\tM_2(u)\dd u$ also has an asymptotic normal
distribution.
Moreover, by the vector-valued version of \refP{P:JS} mentioned in
\refR{R:JS},  the argument in the proof of \refL{Ltm} yields joint
asymptotic normality of the processes $\tmk$ for different $k$; this
uses a straightforward extension of \refL{Lqtm} to quadratic
covariations $[\tM_{k_1}, \tM_{k_2}]_t$. As a result, the first two
terms on the right hand side of \eqref{em4} are jointly normal, and
the $O$ term can be ignored. (The right normalization here is, \cf{}
\refT{Tvar}, to divide by $n^2t\qq(1-nt)^{-9/2}$.) A careful but
rather tedious (even with \maple) calculation of the involved
covariances yields 
$\ts3(t)\sim\AsN(n,\qc3(1/(1-nt))$ with 
$\qc3(x)=96\,{x}^{3}-198\,{x}^{2}+126\,x-24$.
Hence, with $\pc3(x)=x^6\qc3(x)=96\,{x}^{9}-198\,{x}^{8}+126\,x^7-24 x^6$,
\begin{equation}\label{s3t}
  \s3(t)\sim\AsN\Bigpar{\frac{n}{(1-nt)^3},\pc3\unt}.
\end{equation}

We can argue in the same way for $k>3$ too, which leads to the 
recursive formula 
(for all $k\ge2$, \cf{} \eqref{em1} and \eqref{em3} for $k=2$ and 3)
  \begin{multline}\label{emk}
\yx k(t)
=(1-nt)^{-k}\tM_k(t)
+n(1-nt)^{-k}\sum_{j=2}^{k-1}
\binom{k}{j-1}
\\\times
\intot(1-nu)^k\pp{k+2-j}\unu\yx
j(u)\dd u
+\OL\parfracc{nt}{(1-nt)^{2k}}.
  \end{multline}
This yields, by induction, 
\cf{} \eqref{em1} and \eqref{em4} for $k=2$ and 3,
  \begin{multline}\label{emk2}
\yx k(t)
=(1-nt)^{-k}\tM_k(t)
+n(1-nt)^{-k}\sum_{j=2}^{k-1}
\intot \pmkj\unu\tM_j(u)\dd u\\
+\OL\parfracc{nt}{(1-nt)^{2k}},
  \end{multline}
for some polynomials $\pmkj(x)$ having degree at most $k+1-j$ and no
terms of degree $\le1$. The asymptotic joint normality of the
processes $\tM_k$ 
(with a careful count of the degrees of the involved polynomials)
now shows the following extension of \refT{TS2} and \eqref{s3t}.
\begin{theorem}\label{TSK}
  There exist polynomials $\qmkk(x)$ of degree (at most) $2k-3$ such
  that if\/ $1-nt\gg n\qqqw$, then
\begin{equation*}
S_k(t)\sim\AsN\Bigpar{n\ppk\unt,\qmkk\unt},
\qquad k\ge2.
\end{equation*}
Furthermore, this holds jointly for all $k\ge2$, with asymptotic
covariances given by polynomials $\qmkl(x)$ of degree (at most) $2k+2l-3$.
\end{theorem}

We have, for example, 
$\qm_2(x)=2x^5$,
$\qm_3(x)=96\,{x}^{9}-198\,{x}^{8}+126\,x^7-24 x^6$ (as said above),
and
$\qm_{2,3}(x)=12\,{x}^{7}-18\,{x}^{6}+6\,{x}^{5}$.
To find $\qm_k=\qm_{k,k}$ and $\qmkl$ in general by this method seems
quite difficult, although it is in principle possible using computer
algebra.
In the next section we will, by a different method, find the
asymptotics of the covariances of the variables $\skt$.
It is natural to conjecture that these coincide with the asymptotic
covariances in \refT{TSK}, which be general probability theory,
\eg{} \cite[Theorem 5.5.9]{Gut},
is equivalent to uniform square integrability of each of the
standardized variables $(\skt-\E\skt)/\Var(\skt)\qq$ as \ntoo.
This is very plausible (and thus verified for $k=2$ and $3$ by our
calculations of $\qm_2$ and $\qm_3$), but we have so far been unable to
verify it in general, and we leave this as an open problem and conjecture.
(It would suffice to consider the case $nt\le1/2$, say, and show
for example that then $\E|\skt-\E\skt|^4=O(n^2)$.)

\begin{conjecture}
$\qmkl$ equals the polynomial $\hpkl$ defined in \eqref{hpkl}.
\end{conjecture}

\begin{remark}
The purpose of 	introducing $\tsk$ in \eqref{tsk} is that if we argued
directly with $\sk$ and $M_k$, we would obtain an equation similar to
\eqref{emk}, but with $\yx k(u)$ in one of the integrals on the right
hand side.
Thus, to derive the asymptotic normality of $\yx(t)$ from the
asymptotic normality of the processes $M_k$, we would have to invert a
Volterra equation  (also for $k=2$). This is effectively what we do by
introducing $\tsk$.
\end{remark}

\section{The variance again}\label{Svar2}

In \refT{Tvar} we gave a simple upper of the variance for the variance
of 
$\skt$. We shall now, using a more involved argument, find the precise
asymptotics. 

\begin{theorem}
    \label{Tvar2}
For every $k,l\ge2$ and $0\le t<1/n$,
\begin{equation*}
  \Cov\bigpar{\skt,\slt}
=n\hpkl\unt+O\parfrac{nt}{(1-nt)^{2k+2l}},
\end{equation*}
where $\hpkl$ is a polynomial of degree $2k+2l-3$ given by
\begin{equation}
  \label{hpkl}
\hpkl(x)=\pp{k+l}(x)-\frac{\pp{k+1}(x)\pp{l+1}(x)}{x}.
\end{equation}
\end{theorem}

Some polynomials $\hpkl$ are given in \refTab{Tabhpkl}. 
In particular, $\hpxx22(1/y)=2(1-y)/y^5$ and thus
\begin{equation}\label{vars2}
  \Var(\s2(t)) 
= \frac{2n^2t}{(1-nt)^{5}}\biggpar{1  +O\parfracc{1}{n(1-nt)^{3}}}.
\end{equation}
\begin{table}
  \begin{align*}
\hpxx22(x)&=2\,{x}^{5}-2\,{x}^{4}\\
\hpxx33(x)&=96\,{x}^{9}-198\,{x}^{8}+126\,{x}^{7}-24\,{x}^{6}\\
\hpxx44(x)&=10170\,{x}^{13}-34050\,{x}^{12}+43520\,{x}^{11}-26192\,{x}^{10}+7272\,{x}^{9}-720\,{x}^{8}\\
\hpxx32(x)&=12\,{x}^{7}-18\,{x}^{6}+6\,{x}^{5}\\
\hpxx42(x)&=90\,{x}^{9}-190\,{x}^{8}+124\,{x}^{7}-24\,{x}^{6}\\
\hpxx43(x)&=900\,{x}^{11}-2430\,{x}^{10}+2322\,{x}^{9}-912\,{x}^{8}+120\,{x}^{7}  \end{align*}
  \caption{The polynomials $\hpkl(x)$ for $k,l\le4$.} \label{Tabhpkl}
\end{table}

For $1-nt<n\qqqw$, \refT{Tvar2} is a trivial (and uninteresting) consequence of
\refT{Tvar} and the \CSineq, so we assume in the sequel that 
$1-nt\ge n\qqqw$. 
We precede the proof by several lemmas; 
we begin by defining,
extending \eqref{skl},
\begin{equation*}
  \skkm(G)\=\sumx_{i_1,\dots,i_m} \ccc{i_1}^{k_1}\dotsm\ccc{i_m}^{k_m},
\end{equation*}
where $\sumx$ denotes the sum over distinct indices only.
Then, \cf{} \eqref{skl},
  \begin{multline}\label{v1}
  \skkm(G)
=
S_{k_1,\dots,k_{m-1}}(G)\s{k_m}(G)
\\
-S_{k_1+k_m,\dots,k_{m-1}}(G)
\dots
-S_{k_1,\dots,k_{m-1}+k_m}(G),
  \end{multline}
where we subtract $m-1$ terms with $k_m$ added to one of
$k_1,\dots,k_{m-1}$.
For $G=\gnt$ we write $\skkm(t)$ and have the following estimate,
\cf{} \refL{L1}.

\begin{lemma}
  \label{LV1}
For each $k_1,\dots,k_m$ and $1-nt\ge n\qqqw$,
\begin{equation*}
  \E\skkm(t)=n^mp_{k_1}\dotsm p_{k_m}\unt\Bigpar{1+O\parfrac1{n(1-nt)^3}}.
\end{equation*}
\end{lemma}
\begin{proof}
  Immediate by \refT{Tesk}, \eqref{v1} and induction over $m$.
\end{proof}

We write $\sk(t;n)$ when needed to show the number of vertices explicitly.

\begin{lemma}
  \label{Ladd}
For each $k\ge2$ and $1-nt\ge n\qqqw$,
\begin{equation}\label{ladd}
  \E\sk(t;n+1)-\E\sk(t;n)
=\pdk\unt+O\parfrac{t}{(1-nt)^{2k+1}},
\end{equation}
where $\pdk$ is a polynomial of degree $2k-2$ given by
\begin{equation}
  \pdk(x)\=\ppk(x)+(x^2-x)\ppk'(x)= x\qw\pp{k+1}(x).
\end{equation}
\end{lemma}

The formula \eqref{ladd} is, not surprisingly, essentially what 
a formal differentiation of \eqref{tesk} with respect to $n$ would give.

\begin{proof}
  Let $\gnt$ have the components $\cc1,\dots,\cc{K}$.
Add a new vertex and add edges to it with the correct probabilities,
and let $\gD S_k\=\sk(t;n+1)-\sk(t;n)$ be the resulting increase of
$\skt$.
Let $J_i$ be the indicator of the event that there is an edge between
the new vertex and $\cci$. Then
\begin{align*}
  \gD\s2 &=
1+\sum_i2\ccci J_i + \frac12\sumx_{i,j}2\ccci\cccj J_iJ_j,
\\
  \gD\s3 &=
1+\sum_i(3\ccci+3\ccci^2)J_i 
+ \frac12\sumx_{i,j}(3\ccci^2\cccj+3\ccci\cccj^2+6\ccci\cccj )J_iJ_j
\\&
\qquad+ \frac16\sumx_{i,j,k} 6\ccci\cccj\ccc{k}J_iJ_jJ_k,
\end{align*}
and so on. Given the components $\cc1,\cc2,\dots$, the indicators
$J_i$ are independent with $\E J_i=1-e^{-\ccci t}=\ccci t+O(\ccci^2t^2)$.
Hence, for $k=2$, 
using $\ccci t\le nt <1$ to simplify terms like $\ccci^2t^2\cccj^2t^2$,
\begin{equation*}
  \E\bigpar{\gD\s2\mid\gnt}
= 1+2t\s2(t)+O(t^2\s3(t))+t^2S_{2,2}(t)+O(t^3S_{3,2}(t)).
\end{equation*}
Taking the expectation we find, using \refL{LV1},
\begin{equation}\label{eds2}
  \E\gD\s2
= 1+2nt\pp2\unt+(nt)^2\pp2\unt^2+O\parfrac{t}{(1-nt)^5}.
\end{equation}
The same argument applies to every $k$, and yields an expression for
$\E\gD\sk$ where the main terms are of the type
$c(nt)^m\pp{k_1+1}\dotsm\pp{k_m+1}\untx$, where $c$ is a positive
combinatorial constant, $0\le m<k$, $1\le k_i\le k-1$ and
$\sum_ik_i\le k$; the error terms are all $O(1/(n(1-nt)^3)$ of some
such terms.
The main terms are polynomials in $1/(1-nt)$ of degree
$\sum_i(2k_i-1)\le2k-2$, so the result can be written as \eqref{ladd}
for some polynomial $\pdk$.

To identify $\pdk$, fix $y\in(0,1/2)$ and a rational $\eps\in(0,1)$,
consider only $n$ such that $\eps n$ is an integer and let $t=y/n$ and
repeat \eqref{ladd} $\eps n$ times. This yields
\begin{equation*}
  \E\sk(t;(1+\eps)n) -\E\sk(t;n)
=\eps n\Bigpar{\pdk\uny+O(\eps)} +O(\eps),
\end{equation*}
and thus, by \refT{Tesk},
\begin{equation*}
(1+\eps)n\ppk\parfrac{1}{1-(1+\eps)y} -n\ppk\uny
=
\eps n\pdk\uny+O(\eps^2n) 
+O(1).
\end{equation*}
Divide by $n$ and let \ntoo; this gives
\begin{equation*}
\eps \pdk\uny
=
(1+\eps)\ppk\parfrac{1}{1-(1+\eps)y} -\ppk\uny
+O(\eps^2).
\end{equation*}
Divide by $\eps$ and let $\eps\to0$; this gives, with $x=1/(1-y)$,
\begin{equation*}
\pdk(x)
=
\ppk(x)+\frac{y}{(1-y)^2}\ppk'(x)
=\ppk(x)+(x^2-x)\ppk'(x).
\end{equation*}

The final identification of this as $x\qw\pp{k+1}(x)$ follows by
\eqref{pki} proved in \refS{Sborel} below.
Alternatively, the proof of \refT{Tvar2} below and the symmetry of
$\Cov(S_k,S_l)$ shows that
$\pp{k+l}-\pp{k+1}\pd{l}=\pp{k+l}-\pp{l+1}\pd{k}$, and thus, choosing
$l=2$,
$\pdk=\pp{k+1} \cdot\pd2/\pp3$, which yields the formula, since it
follows from \eqref{eds2} that $\pd2(x)=x^2$.
(This thus gives an alternative proof of \eqref{pki}.)
\end{proof}

\begin{proof}[Proof of \refT{Tvar2}]
Let $\cA_n$ and $I_A(t)$ be as in the proof of \refL{L1}. Conditioned on
$I_A(t)=1$, the complement of $A$ is a random graph
equivalent to $\cG(n-|A|,t)$. Thus,
\begin{equation*}
  \begin{split}
\Cov\bigpar{\skt,\slt}
=\E\Bigpar{\sum_{A\in\cA_n}|A|^k I_A(t) \sum_{B\in\cA_n}|B|^l I_B(t)}
-\E\skt\E\slt
\\
=\E\sum_{A\in\cA_n}|A|^{k+l} I_A(t)
+
\E\sum_{A\in\cA_n}|A|^k I_A(t)
 \Bigpar{\sum_{B\cap A=\emptyset}|B|^l I_B(t)-\E\slt}
\\
=\E\s{k+l}(t)
+
\E\sum_{A\in\cA_n}|A|^k I_A(t)
 \Bigpar{\E\s l(t;n-|A|)-\E\s l(t;n)}.
  \end{split}
\end{equation*}
By \refL{Ladd}, for some $\theta\in\oi$,
\begin{equation*}
  \begin{split}
\E\s l(t;n)-\E\s l(n;n-|A|)
=\aaa\pdl\parfrac1{1-nt+\theta\aaa t} + O\parfrac{\aaa t}{(1-nt)^{2l+1}}
\\
=\aaa\pdl\parfrac1{1-nt} 
+ O\parfrac{ t\aaa^2}{(1-nt)^{2l-1}}
+ O\parfrac{t\aaa }{(1-nt)^{2l+1}}.
  \end{split}
\end{equation*}
Consequently,
\begin{multline*}
  \Cov(\skt,\slt)
=
\E\s{k+l}(t)-\E\s{k+1}(t)\pd l\unt 
\\
+ O\Bigpar{\frac{t}{(1-nt)^{2l-1}}\E\s{k+2}(t)}
+O\Bigpar{\frac{t}{(1-nt)^{2l+1}}\E\s{k+1}(t)},
\end{multline*}
and the result follows by \refT{Tesk}.
\end{proof}

In the case $nt\to1$, only the leading term of $\hpkl$ is significant 
in \refT{Tvar2}. Since the
leading term of $\ppk$ is $(2k-5)!!\,x^{2k-3}$, as follows by
\eqref{pki} in \refS{Sborel}, we have the following corollary.

\begin{corollary}
  \label{Cvar2}  
For every $k,l\ge2$, if $nt\to1$ with $1-nt\gg n\qqqw$, then
\begin{equation*}
  \Cov\bigpar{\skt,\slt}
\sim 
\ckl n(1-nt)^{3-2k-2l},
\end{equation*}
with $\ckl\={(2k+2l-5)!! -(2k-3)!!\,(2l-3)!!} $.
\end{corollary}

In particular, under these conditions, 
\begin{align*}
\Var(\s2(t))&\sim2n(1-nt)^{-5},\\
\Var(\s3(t))&\sim 96\,n(1-nt)^{-9},\\
\Var(\s4(t))&\sim 10170\,n(1-nt)^{-11},
\end{align*}
\cf{} \refTab{Tabhpkl} and \eqref{vars2}.

\begin{proof}[Proof of \refT{T2}, asymptotic variance]  
Immediate from \refT{TS2}, see \eqref{vars2}.
\end{proof}

\section{The Borel distribution}\label{Sborel}

Let $T(z)$
be the \emph{tree function} 
\begin{equation*}
T(z)\=\sum_{j=1}^\infty\frac{j^{j-1}z^j}{j!},
\qquad |z|\le e^{-1},
\end{equation*}
and recall the well-known formulas $T(z)e^{-T(z)}=z$
($|z|\le e^{-1}$),
$T(\ga e^{-\ga})=\ga$ ($0\le\ga\le1$),
and
\begin{equation}
\label{tprim}
T'(z)=\frac{T(z)}{z\bigpar{1-T(z)}}.
\end{equation}

A random variable $\ba$ has the
\emph{Borel distribution} $\Bo(\gl)$ 
with parameter $\gl\in\oi$
if
\begin{equation}
\label{borel}
\P(\ba=j)=\frac{j^{j-1}}{j!}\gl^{j-1}e^{-j\gl}
=\frac1{T(\gl e^{-\gl})} \frac{j^{j-1}}{j!}(\gl e^{-\gl})^j,
\qquad j=1,2,\dots
\end{equation}
The \pgf{} of the Borel distribution is thus
\begin{equation}
\label{borelg}
\E z^{\ba}=\sum_{l=1}^\infty\P(\ba=l)z^l
=\frac{T(\gl e^{-\gl}z)}{T(\gl e^{-\gl})}
=\frac{T(\gl e^{-\gl}z)}{\gl }.
\end{equation}
It is well-known that $\Bo(\gl)$ is 
the distribution of the total progeny of a Galton--Watson branching
process where each individual has $\Po(\gl)$ children; for this and related
results, see \eg{}
\cite{Borel,Otter,Kendall,Tanner,Takacs:Comb,Dwass,Takacs:ballots,Pitman:enum,
  JohnsonKK}.

Now consider $\gnp$ with $p=\gl/n$ for a fixed $\gl<1$, and let $\cc v$
be the component containing a fixed vertex $v$. It is
easily seen that as \ntoo, for every fixed $j\ge1$,
$\P(\ccc v=j)\to\P(\ba=j)$ given by \eqref{borel},
either by the usual branching process approximation and the result
just quoted, or by 
a direct estimation of the probability, using 
Cayley's formula for the number of trees of order $j$
and
the fact that \whp{} the component $\cc v$ is a tree. 
In other words, $\ccc v\dto\ba$. For any integer $m$, the moment
$\E\ccc v^m=\E\s{m+1}(\gnp)/n$, and \refT{Tesk} shows, with $t=-\log(1-p)$
and thus $nt\to\gl$, that
\begin{equation*}
\E\ccc v^m
=\frac{\E\s{m+1}(G(n,\gl/n))}{n}
\to\pp{m+1} \ul.  
\end{equation*}
Since thus $\ccc v$ converges in distribution and all moments converge
(to finite limits),
the moments have to converge to the moments of the limit
distribution. We have thus shown the following.

\begin{theorem}\label{TBorel}
  The polynomials $\ppk$ describe the moments of the Borel
  distribution $\Bo(\gl)$ by the formula
\begin{equation*}
\E\ba^m
=\pp{m+1} \ul,
\qquad m\ge1.
\end{equation*}
\end{theorem}

For example, as is well-known,
$\E \ba=(1-\gl)\qw$ and $\E \ba^2=(1-\gl)^{-3}$.

\begin{remark}
  By \refT{TBorel}, \refC{Cvar} can be written 
  \begin{equation*}
\skt\simp n\E \bax{nt}^{k-1},
\qquad 1-nt\gg n\qqqw.	
  \end{equation*}
This is not surprising since we have $S_k(G)=\sum_v\ccc v^{k-1}$, 
and we
expect only a weak dependence between the components $\cc v$ in this
range, so this is a kind of law of large numbers.
\end{remark}

Let, \cf{} \eqref{borelg}, for $|t|$ small enough,
  \begin{equation}\label{mgf}
\psitl=\E e^{t\ba}=\sum_{m=0}^\infty\frac{t^m}{m!}\E\ba^m
= \frac{\TT}{\gl}
  \end{equation}
be the moment generating function of $\ba\sim\Bo(\gl)$. 
The moments of $\ba$ can be obtained by differentiation of $\psitl$ at $t=0$.

\begin{lemma}
  \label{LBomom}
For each $m\ge0$ there exists a polynomial $\pr m$ such that
\begin{equation}
  \label{rm}
\frac{\dd^m}{\dd t^m} \psitl
=\frac{\TT}{\gl}\prm\parfrac{1}{1-\TT}.
\end{equation}
We have $\pr0(x)=1$, $\pr0(x)=x$, and
\begin{equation}
  \label{rmi}
\pr{m+1}(x)=x\prm(x)+(x^3-x^2)\prm'(x),
\qquad m\ge0.
\end{equation}
\end{lemma}

\begin{proof}
  For $m=0$, \eqref{rm} is just \eqref{mgf}.

Suppose that \eqref{rm} holds for some $m\ge0$. Then, by the chain
rule and \eqref{tprim}, with $T=\TT$,
\begin{equation*}
  \begin{split}
\frac{\dd^{m+1}}{\dd t^{m+1}} \psitl
&=\frac{\dd}{\dd T}
\Bigpar{\frac{T}{\gl}\prm\parfrac{1}{1-T}}	
\cdot\frac{T}{1-T}
\\&
=\frac1\gl\frac{T}{1-T}\prm\parfrac{1}{1-T}
+\frac{T^2}{\gl(1-T)^3}\prm'\parfrac{1}{1-T}
\\&
=\frac{T}\gl
\lrpar{\frac{1}{1-T}\prm\parfrac{1}{1-T}
+\Bigpar{\frac1{(1-T)^3}-\frac1{(1-T)^2}}\prm'\parfrac{1}{1-T}},
  \end{split}
\end{equation*}
which verifies \eqref{rm} for $m+1$ with $\pr{m+1}$ given by \eqref{rmi}.
\end{proof}

Since $\pr1(x)=x$, it follows from \eqref{rmi} by induction that
$\prm$ has degree $2m-1$ for $m\ge1$.

Setting $t=0$ in \eqref{rm} yields
\begin{equation*}
\E\ba^m
=\frac{\dd^{m}}{\dd t^{m}} \psitl\bigrestr{t=0}
=\pr{m} \ul,
\qquad m\ge0.
\end{equation*}
Consequently, \refT{TBorel} shows that 
\begin{equation}
 \prm(x)=\pp{m+1}(x), 
\qquad m\ge1.
\end{equation}
In particular, \eqref{rmi} yields the simple linear recursion
\begin{equation}
  \label{pki}
\pp{k+1}(x)=x\ppk(x)+(x^3-x^2)\ppk'(x),
\qquad k\ge2.
\end{equation}
It is evident from \eqref{pki} and induction that, for $k\ge2$, the
leading term of $\ppk$ is $(2k-5)!!\,x^{2k-3}$ 
(with the standard interpretation $(-1)!!=1$)
and that 
for $k\ge3$, 
the lowest
order non-zero term is $(-1)^{k-1}(k-2)!\,x^k$, see \refTab{Tabpk}.

\begin{remark}
The quadratic recursion \eqref{pk1} can be seen to be equivalent to 
the quadratic partial differential equation
\begin{equation*}
\frac{\partial}{\partial\gl} \psitl
=\bigpar{\psitl-1}
\frac{\partial}{\partial t} \psitl,
\end{equation*}
while the linear recursion \eqref{pki} is equivalent to 
the linear partial differential equation
\begin{equation*}
\frac{\partial\psi}{\partial t} (t;\gl)
=\ulx\psitl+\frac{\gl}{1-\gl}\frac{\partial\psi}{\partial \gl}(t;\gl).
\end{equation*}
\end{remark}

\begin{remark}
  By \refT{TBorel},
the recursion \eqref{pk1} can be written
\begin{equation*}
\frac{\dd}{\dd\gl}\E\ba^{k-1}
=(1-\gl)\qww\ppk'\ul
=\frac1{2}\sumlki\binom kl \E\ba^{l}\E\ba^{k-l},
\end{equation*}
or, if $\ba'$ and $\ba''$ are independent copies of $\ba$, 
using $\frac{\dd}{\dd\gl}\P(\ba=j)=\bigpar{\frac{j-1}{\gl}-j}\P(\ba=j)$
from \eqref{borel},
\begin{equation*}
  \begin{split}
  \E(\ba'+\ba'')^k
&=2\E\ba^k+2\frac{\dd}{\dd\gl}\E\ba^{k-1}
\\&
=\sum_{j=1}^\infty\P(\ba=j)\cdot
 \Bigpar{2j^k+2j^{k-1}\bigpar{\frac{j-1}{\gl}-j}}	
\\&
=\sum_{j=1}^\infty\P(\ba=j)\cdot\frac{2(j-1)}{j\gl}j^k,
  \end{split}
\end{equation*}
which is equivalent to the well-known formula 
\begin{equation*}
  \P(\ba'+\ba''=j)
=\frac{2(j-1)}{j\gl}\P(\ba=j)
=2\frac{j^{j-3}}{(j-2)!}\gl^{j-2}e^{-j\gl},
\qquad j\ge2;
\end{equation*}
see \eg{} \cite{Tanner,Takacs:ballots,Pitman:enum,JohnsonKK} and note that
$\ba'+\ba''$ can be seen as the total progeny of a Galton--Watson process
with $\Po(\gl)$ offspring started with 2 individuals, or as the limit
distribution of $|\cc v\cup\cc w|$ if $\cc v$ and $\cc w$ are the
components containg two given vertices in $G(n,\gl/n)$.
\end{remark}

\begin{remark}
  The \emph{cumulants} $\gk_m$ of the Borel distribution $\Bo(\gl)$
  are 
the Taylor coefficients of $\log\psitl$ at $t=0$
(times $m!$).
Since $T(z)=ze^{T(z)}$, \eqref{mgf} yields
  \begin{equation*}
\log\psitl= \TT-\gl+t=\gl\psitl-\gl+t,
  \end{equation*}
and thus
\begin{equation*}
  \gk_m(\ba)
=\frac{\dd^{m}}{\dd t^{m}}\log \psitl\bigrestr{t=0}
=\gl\E\ba^m
=\gl\pp{m+1}\ul,
\qquad m\ge2,
\end{equation*}
while, of course, $\gk_1(\ba)=\E\ba=(1-\gl)\qw$.
\end{remark}

We can  interpret the asymptotic covariances and the polynomials 
$\hpkl$ in \refS{Svar2} by introducing the 
\emph{size-biased Borel distribution} $\hba$ defined by
\begin{equation}
  \P(\hba=j)=\frac{j\P(\ba=j)}{\E\ba}
=(1-\gl)\frac{j^{j}}{j!}\gl^{j-1}e^{-j\gl}.
\end{equation}
Then 
\begin{equation}
  \E\hba^m=\E\ba^{m+1}/\E\ba
=(1-\gl)\pp{m+2}\ul,
\qquad m\ge0,
\end{equation}
and thus, by \eqref{hpkl},
\begin{equation}
  \hpkl\unt = \untx \Cov\bigpar{\hbnt^{k-1},\hbnt^{l-1}}.
\end{equation}

Hence, by \refT{Tvar2},
the random variables $n\qqw(1-nt)\qq\skt$, $k\ge2$, 
have asymptotically the same covariance structure as $\hbnt^{k-1}$.

\appendix

\section{The supercritical case}\label{Asuper} 

Consider $\gnp$ with $np-1\gg n\qqqw$. It is well-known, see \eg{}
\cite[Chapter 5]{JLR}, that \whp{} $\gnp$ has a unique giant
component. More precisely, there is a deterministic function $\rho>0$
on $(1,\infty)$ such that, if the components $\cc1,\cc2,\dots$ of
\gnp{} are ordered with $\ccc1\ge\ccc2\ge\dots$, then 
$\ccc1\simp n\rho(np)\gg n\qqa$, while $\ccc2=\op(n\qqa)$.
The function $\rho(\gl)$ is the survival probability of a
Galton--Watson branching process with $\Po(\gl)$ offspring, and is
given by the equation
\begin{equation}\label{rho}
  \rho(\gl)=1-e^{-\gl\rho(\gl)}.
\end{equation}

The largest component is thus much larger than the others, and it
turns out that it dominates all other terms in the sums $\sk$.
We write in this appendix $\sknp$ for $\sk(\gnp)$, and continue to
let $\cc1$ denote the largest component of \gnp{}.

\begin{theorem}
  \label{Tsuper}
If $np-1\gg n\qqqw$, then for every $k\ge2$,
\begin{equation*}
  \sknp =\ccc1^k+\Op\parfracc{n}{(np-1)^{2k-3}}
\simp\ccc1^k
\simp\bigpar{n\rho(np)}^k.
\end{equation*}
\end{theorem}

In particular, then $\sus(\gnp)\simp n\rho(np)^2$.
We first prove a technical lemma.

\begin{lemma}\label{Lsuper}
  There exists a function $\ga:(1,\infty)\to(0,1)$ such that the
  following holds, for some $c>0$:
  \begin{romenumerate}
\item\label{Lsuperx} For any $p=p(n)$ with $np-1\gg n\qqqw$,
\whp{} $\ccc1>\ga(np) n$.
\item\label{Lsupera}
If $1<\gl\le2$, then $\gl(1-\ga(\gl)) \le 1-c(\gl-1)$.
\item\label{Lsuperb}
If $\gl\ge2$, then $\gl(1-\ga(\gl)) \le 1-c$.
\item\label{Lsuperm}
For each $m\ge0$, $1-\ga(\gl)=O(\gl^{-m})$.
  \end{romenumerate}
\end{lemma}

\begin{proof}
For any fixed $M>1$, we can take $\ga(\gl)=(1-\eps)\rho(\gl)$ for
$1<\gl\le M$, if $\eps$ is sufficiently small. 
This choice satisfies \ref{Lsuperx}, 
in this range
\ref{Lsuperm} is trivial, and 
it is easily seen that \ref{Lsupera} and \ref{Lsuperb} follow
(provided $\eps$ is small enough)
from the
facts that $\rho(\gl)\sim 2(\gl-1)$ and 
$\gl(1-\rho(\gl))=1-(\gl-1)+O(\gl-1)^2$ as $\gl\downto1$,
and $\gl(1-\rho(\gl))<1$ for $\gl>1$. (All three are easily verified by
writing \eqref{rho} as $\gl=-\log(1-\rho)/\rho$.)

For large $\gl$, we argue as follows. Take $\gam<\rho(2)$. Thus, \whp{}
$G(n,2/n)$ has a giant component of order at least $\gam n$. For $\gl=np>2$,
construct $G(n,p)$ by the usual two-round method:
first take $G(n,2/n)$ and then add further edges independently
in a second round with probabilities $p-2/n$ (or, to be precise,
$(np-2)/(n-2)>p-2/n$). If we obtain a component of order at least
$\gam n$
in the first round, then the probability that a given vertex will \emph{not}
be joined to this component in the second round is less than
$\exp(-\gam n(p-2/n))=\exp(2\gam-\gam\gl)$. Hence, \whp{} the number of such
vertices is less than $n\exp(2\gam-\gam\gl/2)$; for $\gl=O(1)$ by
concentration of the binomial distribution and for $\gl\to\infty$ by
Markov's inequality. Consequently, there is \whp{} a component with
more than $n-n\exp(2\gam-\gam\gl/2)$ vertices; hence \ref{Lsuperx} holds
with $\ga(\gl)=1-\exp(2\gam-\gam\gl/2)$. This $\ga$ satisfies \ref{Lsuperm}
too, and \ref{Lsuperb} for large enough $\gl$. We thus can use this $\ga$
for $\gl> M$ for some large $M$, and the first construction for
smaller $\gl$.
\end{proof}

\begin{proof}[Proof of \refT{Tsuper}]
  Let $\ga=\ga(np)$ be as in \refL{Lsuper}, and use the notation of the proof
  of \refL{L1}. 
Let $N$ be the number of components of size $>\ga n$ in \gnp{}
(thus \whp{} $N\ge1$ by \refL{Lsuper}\ref{Lsuperx}), and let
\begin{equation*}
  Z_k\=\sum_{|A|>\ga n}\sum_{B\cap A=\emptyset} |B|^k I_AI_B.
\end{equation*}
Then
\begin{equation}\label{a4}
\E  Z_k\=\E\sum_{|A|>\ga n} I_A\E \sk(n-|A|,p)
\le \E N\E\sk\nap.
\end{equation}

If $1<np\le2$, then by \refL{Lsuper}\ref{Lsupera}, 
$(\nna)p\le np(1-\ga)\le 1-c(np-1)$, and thus by \refL{Lesk},
\begin{equation*}
  \E\sknap = O\parfracc{\nna}{(np-1)^{2k-3}}
= O\parfracc{n}{(np-1)^{2k-3}}.
\end{equation*}
If instead $np>2$, then by \refL{Lsuper}\ref{Lsuperb}, 
$(\nna)p\le np(1-\ga)\le 1-c$, and thus by Lemmas \ref{Lesk} and
\ref{Lsuper}\ref{Lsuperm}, with $m=2k-3$,
\begin{equation*}
  \E\sknap = O(\nna)
= O\parfracc{n}{(np)^{2k-3}}.
\end{equation*}
Hence, for all $np$, 
\begin{equation}\label{a7}
  \E\sknap 
= O\parfracc{n}{(np-1)^{2k-3}}
=o\bigpar{n^{2k/3}}.
\end{equation}

Note first that $Z_k\ge N(N-1)\ga^k$. Hence, by \eqref{a4} and \eqref{a7},
\begin{equation*}
  \E N(N-1)\le \ga^{-k}\E Z_k \le o(\E N n^{2k/3}\ga^{-k})=o(\E N).
\end{equation*}
Since $N\le 1+N(N-1)$, it follows that $\E N(N-1)=o(1)$ and $\E N=O(1)$;
hence \eqref{a4} and \eqref{a7} yield $\E Z_k=O(n/(np-1)^{2k-3})$.
By \refL{Lsuper}\ref{Lsuperx}, \whp{} $\ccc1>\ga n$; in this case, 
$\ccc1^k\le\sknp\le\ccc1^k+ Z_k$, and the result follows.
\end{proof}

\section{The critical case}\label{Acritical} 

The critical case is $np=1+O(n\qqqw)$. By considering subsequences, it
suffices to consider the case $n\qqq(np-1)\to \tau$ for some
$\tau\in(-\infty,\infty)$, \ie, $np=1+(\tau+o(1)) n\qqqw$.

We continue to use the notations of \refApp{Asuper}.
It is well-known that in the critical case, $\ccc1$ is of the order
$n\qqa$, in the sense that $\ccc1/n\qqa$ converges in distribution to
some non-degenerate random variable, and the same holds for $\ccc2$,
$\ccc3$, \dots
Moreover,
\citet{Aldous} has shown that, with notations as in \refApp{Asuper},
the sequence $(n\qqaw\ccc1,n\qqaw\ccc2,\dots)$ (extended by an
infinite number of 0's) 
converges in distribution to a certain random sequence
$(\ct(1),\ct(2),\dots)$ that can be described as the sequence of
excursion lengths of a certain reflecting Brownian motion with
inhomogeneous drift (depending on $\tau$) that is defined in
\cite{Aldous}.
The convergence is in the $\ell^2$-topology, and thus immediately
implies convergence of the sums of squares. Moreover, convergence in
$\ell^2$ implies convergence in $\ell^k$ for every $k\ge2$, and thus
we also have convergence of the sums of $k$th powers.
Consequently, 
\begin{theorem}
If $np=1+(\tau+o(1)) n\qqqw$ with $-\infty<\tau<\infty$, then for
every $k\ge2$,
\begin{equation*}
  n^{-2k/3}\sk(n,p)\dto W_k\=\sum_i \ct(i)^k.
\end{equation*}
\end{theorem}

Note that we here have limits that are non-degenerate random variables
and not constants, unlike the subcritical and supercritical cases
where $\sk(n,p)\simp a_n$ for a suitable sequence $a_n$.

\begin{remark}
\citet{SJ179} give a related description 
of the limit of the component
sizes as a point process $\Xitau$ on $(0,\infty)$. It follows that we
also have $W_k=\int_0^\infty x^k\dd\Xitau(x)$, and thus
$\E W_k=\int_0^\infty x^k\dd\Latau(x)$, where $\Latau$ is the
intensity of $\Xitau$ given in \cite[Theorem 4.1]{SJ179}.  
\end{remark}

\newcommand\AAP{\emph{Adv. Appl. Probab.} }
\newcommand\JAP{\emph{J. Appl. Probab.} }
\newcommand\JAMS{\emph{J. \AMS} }
\newcommand\MAMS{\emph{Memoirs \AMS} }
\newcommand\PAMS{\emph{Proc. \AMS} }
\newcommand\TAMS{\emph{Trans. \AMS} }
\newcommand\AnnMS{\emph{Ann. Math. Statist.} }
\newcommand\AnnPr{\emph{Ann. Probab.} }
\newcommand\CPC{\emph{Combin. Probab. Comput.} }
\newcommand\JMAA{\emph{J. Math. Anal. Appl.} }
\newcommand\RSA{\emph{Random Struct. Alg.} }
\newcommand\ZW{\emph{Z. Wahrsch. Verw. Gebiete} }
\newcommand\DMTCS{\jour{Discr. Math. Theor. Comput. Sci.} }

\newcommand\AMS{Amer. Math. Soc.}
\newcommand\Springer{Springer}
\newcommand\Wiley{Wiley}

\newcommand\vol{\textbf}
\newcommand\jour{\emph}
\newcommand\book{\emph}
\newcommand\inbook{\emph}
\def\no#1#2,{\unskip#2, no. #1,} 
\newcommand\toappear{\unskip, to appear}

\newcommand\webcite[1]{
\texttt{\def~{{\tiny$\sim$}}#1}\hfill\hfill}
\newcommand\webcitesvante{\webcite{http://www.math.uu.se/~svante/papers/}}
\newcommand\arxiv[1]{\webcite{arXiv:#1.}}

\def\nobibitem#1\par{}

\end{document}